\documentclass[a4paper,reqno,11pt, oneside]{amsart}
\usepackage[pdftex]{graphicx,xcolor}
\usepackage{amssymb}
\usepackage{amstext}
\usepackage{amsmath}
\usepackage{amscd}
\usepackage{amsthm}
\usepackage{amsfonts}
\usepackage{enumerate}
\usepackage{caption}
\usepackage{color}
\usepackage{here} 
\usepackage{bm} 
\usepackage{tikz}
\usepackage{calc}
\usetikzlibrary{decorations.markings}
\usetikzlibrary{positioning}
\tikzstyle{vertex}=[circle ,draw, inner sep=0pt, minimum size=6pt]

\makeatletter
  
  \@addtoreset{equation}{section}
\makeatother

\newcommand{\calB}{\mathcal{B}}

\newcommand{\ZZ}{\mathbb{Z}}

\newcommand{\kk}{\Bbbk}

\def\opn#1#2{\def#1{\operatorname{#2}}} 
\opn\conv{conv} \opn\mut{mut} \opn\GL{GL} \opn\cone{cone} \opn\ini{in} \opn\NF{NF} \opn\deg{deg}
\opn\Graph{Graph} \opn\sign{sign} \opn\mat{Mat} \opn\rank{rank} \opn\type{type} \opn\reg{reg} \opn\core{core}
\opn\indmat{im} \opn\Indeg{Indeg} \opn\star{star} \opn\link{link} \opn\Tor{Tor} \opn\MNF{MNF} \opn\dep{depth} \opn\pdim{pdim}
\newcommand{\gdd}{\Graph_{\dep,\dim}}
\newcommand{\gddcc}{\Graph_{\dep,\dim}^\mathrm{chordal}}

\newtheorem{thm}{Theorem}[section]
\newtheorem{cor}[thm]{Corollary}
\newtheorem{lem}[thm]{Lemma}
\newtheorem{prop}[thm]{Proposition}

\newtheorem{q}[thm]{Question}

\theoremstyle{definition}

\newtheorem{ex}[thm]{Example}

\theoremstyle{remark}
\newtheorem{rem}[thm]{Remark}

\begin{document}

\title{Behaviors of pairs of dimensions and depths of edge ideals}
\author{Akihiro Higashitani}
\author{Akane Kanno}
\author{Ryota Ueji}

\address[A. Higashitani]{Department of Pure and Applied Mathematics, Graduate School of Information Science and Technology, Osaka University, Suita, Osaka 565-0871, Japan}
\email{higashitani@ist.osaka-u.ac.jp}
\address[A. Kanno]{Department of Pure and Applied Mathematics, Graduate School of Information Science and Technology, Osaka University, Suita, Osaka 565-0871, Japan}
\email{u825139b@ecs.osaka-u.ac.jp}
\address[R. Ueji]{Department of Pure and Applied Mathematics, Graduate School of Information Science and Technology, Osaka University, Suita, Osaka 565-0871, Japan}
\email{u136547i@alumni.osaka-u.ac.jp}

\subjclass[2020]{
Primary: 13D02, 
Secondary: 13F55, 05E40, 05C70.} 
\keywords{Edge ideals, dimension, depth, projective dimension.}

\maketitle

\begin{abstract} 
Edge ideals of finite simple graphs $G$ on $n$ vertices are the ideals $I(G)$ of the polynomial ring $S$ in $n$ variables 
generated by the quadratic monomials associated with the edges of $G$. 
In this paper, we consider the possible pairs of dimensions and depths of $S/I(G)$ for connected graphs with a fixed number of vertices. 
We discuss such pairs in the case where dimension is relatively large. 
As a corollary, we completely determine the pairs for connected graphs with small number of vertices. 
We also study the possible pairs for connected chordal graphs. 
\end{abstract}

\section{Introduction}

Monomial ideals are one of the most well-studied objects in the area of combinatorial commutative algebra. 
Especially, edge ideals of graphs are of particular interest in this area. 
On the other hand, dimensions and depths are fundamental invariants on graded (or local) rings. 
The goal of this paper is the investigation of the behaviors of pairs of dimensions and depths for the quotient rings by edge ideals. 

\medskip

Throughout the present paper, we only treat finite simple graphs, so we omit ``finite simple''. 
Let $\Graph(n)$ denote the set of all connected graphs on the vertex set $[n]:=\{1,2,\ldots,n\}$. 
Let $S=\kk[x_1,\ldots,x_n]$ be a polynomial ring in $n$ variables over a field $\kk$. 
The {\em edge ideal} $I(G)$ of $G \in\Graph(n)$ is defined by $I(G)=( x_ix_j : \{i,j\} \in E(G)) \subset S$. 
Invariants on edge ideals have been well studied in combinatorial commutative algebra. 
Especially, the graded Betti numbers of the edge ideals of graphs have been considered in various viewpoints. 
Since the (Castelnuovo--Mumford) regularity and the projective dimension of edge ideals can be obtained from the information on Betti numbers, 
we can also obtain some characterizations or bounds on them. For example, see \cite{GV, HV, HKKMV, HKM, K, Kimura, Z}. 
Note that the depths of edge ideals and the projective dimensions are equivalent invariants in some sense (see Subsection~\ref{subsec:AB}). 
For more information on edge ideals, see, e.g., \cite[Section 9]{HerzogHibi}. 

In this paper, we focus on the relationship between dimensions and depths of edge ideals for connected graphs. 
The study of their relationship is initiated in \cite[Section 2]{HKKMV} as far as the authors know. 
Let us briefly recall what is obtained there. Given a positive integer $n$, let 
$$\gdd(n):=\left\{(\dep S/I(G), \dim S/I(G)) : G \in \Graph(n)\right\},$$
where $\dim G$ (resp. $\dep G$) denotes the dimension (resp. depth) of $S/I(G)$. 
On the other hand, let 
$$C^*(n):=\left\{(a,b) \in \ZZ^2 : 1 \leq a \leq b \leq n-1, \; a \leq b+1 - \left\lceil\frac{b}{n-b}\right\rceil\right\}.$$
Then the following is proved: 
\begin{thm}[{\cite[Theorem 2.8]{HKKMV}}]\label{thm:inc}
Let $n \geq 2$. Then we have $C^*(n) \subset \gdd(n)$. 
\end{thm}

By taking this theorem into account, the following question naturally arises: 
\begin{q}\label{qqq}
Does the equality $C^*(n) = \gdd(n)$ hold? 
\end{q}
In \cite[Example 2.2]{HKKMV}, it is mentioned that $\gdd(n)$ is computed by using computer. 
As a consequence of the computations, we can verify that $C^*(n)=\gdd(n)$ holds when $n \leq 9$.

In this paper, we develop more theoretical considerations. 
The first main result of this paper is the following: 
\begin{thm}[{See Remark~\ref{rem:star}, Propositions~\ref{prop:n-2} and \ref{prop:n-3}}]\label{main1}
Let $(a,b) \in \gdd(n)$. If $b \geq n-3$, then $(a,b) \in C^*(n)$. 
\end{thm}
This theorem enables us to show the equality $C^*(n)=\gdd(n)$ when $n$ is relatively small as follows: 
\begin{cor}\label{cor:small}
The equality $C^*(n)=\gdd(n)$ holds if $n \leq 12$. 
\end{cor}

\medskip

For a graph $G$, we say that $G$ is \textit{chordal} if any cycle in $G$ of length at least $4$ has a chord. 
Let 
\begin{align*}
\gddcc(n)=\left\{(\dep S/I(G), \dim S/I(G)) : G \in \Graph(n), \; \text{$G$ is chordal}\right\}.
\end{align*}
Namely, we restrict $\gdd(n)$ into chordal graphs. 

The second main result of this paper is the following: 
\begin{thm}\label{main2}
Let $n \geq 2$. Then we have $C^*(n) = \gddcc(n)$. 
\end{thm}
Note that in \cite[Section 3]{HKKMV}, the possible pairs of dimensions and depths of the edge ideals of Cameron--Walker graphs are determined. 
See \cite[Theorem 3.15]{HKKMV}. We can observe that such pairs for Cameron--Walker graphs are quite small compared with $C^*(n)$. 

\medskip

A brief structure of this paper is as follows. 
In Section~\ref{sec:pre}, we recall some fundamental materials used for the proofs of Theorems~\ref{main1} and \ref{main2}. 
In Section~\ref{sec:main1}, we prove Theorem~\ref{main1} and Corollary~\ref{cor:small}. 
In Section~\ref{sec:main2}, we prove Theorem~\ref{main2}. 

\medskip

\subsection*{Acknowledgements} 
The first named author is partially supported by JSPS Grant-in-Aid for Scientists Research (C) JP20K03513. 
The second named author is partially supported by Grant-in-Aid for JSPS Fellows JP21J21603. 

\bigskip

\section{Preliminaries}\label{sec:pre}

In this section, we prepare some terminologies, notions and statements for the proofs of our main theorems. 

\subsection{Terminologies on graphs}

Let $G$ be a graph on the vertex set $V(G)$ with the edge set $E(G)$. 
\begin{itemize}
\item We call a subset $W$ of $V(G)$ an {\em independent set} (resp. {\em clique}) if $\{v,w\} \not\in E(G)$ (resp. $\{v,w\} \in E(G)$) for any $v,w \in W$. Let 
\begin{align*}
\alpha(G)=\max\{|W| : W \text{ is an independent set of }G\}, 
\end{align*}
called the {\em independence number} of $G$. 
It is well known that $$\dim S/I(G)=\alpha(G).$$
\item Let $M \subset E(G)$ be a subset of the edge set. We say that $M$ is an {\em induced matching} of $G$ if $M$ satisfies that 
\begin{itemize}
\item $e \cap e' = \emptyset$ for any $e,e' \in M$ with $e \neq e'$; and 
\item there is no $e'' \in E(G)$ such that $e \cap e'' \neq \emptyset$ and $e' \cap e'' \neq \emptyset$. 
\end{itemize}
\item For a vertex $v$ of $G$, let $N(v)=\{w \in V : \{v,w\} \in E(G)\}$, let $N[v]=N(v) \cup \{v\}$ and let $\deg v=|N(v)|$. 
\item For a vertex $v$ of $G$, let $G \setminus v$ denote the induced subgraph of $G$ by $V(G) \setminus \{v\}$. 
Similarly, for $U \subset V(G)$, let $G \setminus U$ denote the induced subgraph of $G$ by $V(G) \setminus U$. 
\item We say that $C \subset V(G)$ is a {\em vertex cover} if $C \cap e \neq \emptyset$ for any $e \in E(G)$. 
We call a vertex cover $C$ {\em minimal} if $C \setminus \{v\}$ is not a vertex cover of $G$ for any $v \in C$. 
\item A {\em perfect elimination ordering} of $G$ is an ordering $v_n,\ldots,v_1$ of the vertices of $G$ such that 
$\{v_1,\ldots,v_{k-1}\} \cap N(v_k)$ forms a clique for each $k=2,\ldots,n$. 
It is known that $G$ is chordal if and only if $G$ has a perfect elimination ordering (\cite{D}). 
\end{itemize}

\medskip

In this paper, given a graph $G$, we use the notation $\dim G$, $\dep G$ and $\pdim G$ 
instead of $\dim S/I(G)$, $\dep S/I(G)$ and $\pdim S/I(G)$ for short, respectively, where $\pdim$ denotes the projective dimension.

\subsection{Bouquets and projective dimensions of $S/I(G)$}

We recall the terminology and the notation used in \cite{Kimura}. 
The graph $B$ on the vertex set $V(B)=\{w,z_1,\ldots,z_d\}$ with the edge set $\{\{w,z_i\} : i=1,\ldots,d\}$ is called a {\em bouquet}. 
Namely, a bouquet is nothing but a star graph. 
A vertex $w$ is called a {\em root}, the vertices $z_i$ are called {\em flowers} and the edges $\{w,z_i\}$ are called {\em stems} of $B$. 
These words were originally used in \cite[Definition 1.7]{Z}. 

Given a graph $G$, we call $B$ a bouquet of $G$ if $B$ is a (not necessarily induced) subgraph of $G$. 
Let $\calB=\{B_1,\ldots,B_p\}$ be a set of bouquets of $G$. We define the following: 
\begin{align*}
F(\calB)&:=\bigcup_{i=1}^p\{z \in V(B_i) : z \text{ is a flower of }B_i\}; \\
R(\calB)&:=\bigcup_{i=1}^p\{z \in V(B_i) : z \text{ is a root of }B_i\}; \\
S(\calB)&:=\bigcup_{i=1}^p\{s \in E(B_i) : s \text{ is a stem of }B_i\}. 
\end{align*}

Let $G$ be a graph and let $\calB=\{B_1,\ldots,B_p\}$ be a set of bouquets of $G$. 
Assume that $V(B_i) \cap V(B_j)=\emptyset$ holds for any $1 \leq i<j\leq p$. 
\begin{itemize}
\item We say that $\calB$ is {\em strongly disjoint} 
if there are stems $s_i \in E(B_i)$ for each $i$ such that $\{s_1,\ldots,s_p\}$ forms an induced matching of $G$. 
\item We say that $\calB$ is {\em semi-strongly disjoint} if $R(\calB)$ forms an independent set. 
\end{itemize}
Note that if $\calB$ is strongly disjoint, since any stem in $B_i$ is incident to the root of $B_i$, 
we see that any two roots should not be adjacent. Namely, $\calB$ becomes semi-strongly disjoint. 

\begin{thm}[{\cite[Theorem 3.1]{Kimura}}]\label{thm:K1}
Given any graph $G$, we have 
$$\pdim G \geq \max\{|F(\calB)| : \calB \text{ is a strongly disjoint set of bouquets of }G\}.$$
\end{thm}
By Theorem~\ref{thm:K1}, we see that given a graph $G$, we have \begin{align}\label{eq:deg}\deg v \leq \pdim G\text{ for any }u \in V(G)\end{align} 
since $N(u) \cup \{u\}$ forms a (set of a single) bouquet with a root $u$.  

\begin{thm}[{\cite[Theorem 5.3, Corollary 5.6]{Kimura}}]\label{thm:K2}
Let $G$ be chordal. Then we have 
\begin{align*}\pdim G &= \max\{|F(\calB)| : \calB \text{ is a semi-strongly disjoint set of bouquets of }G\} \\
&=\max\{ |C| : C \text{ is a minimal vertex cover of }G\}.\end{align*}
\end{thm}

\subsection{Auslander--Buchsbaum formula}\label{subsec:AB}
Let $G$ be a graph on $[n]$. Then $S/I(G)$ is an $S$-module. Since $S$ is a polynomial ring, $\pdim G < \infty$ holds. 
Hence, by the Auslander--Buchsbaum formula (see, e.g., \cite[Theorem 1.3.3]{BH}), we have $$ \pdim G + \dep G =n.$$ 
We will often use this in the sequel. 

\subsection{On the sets $C^*(n)$ and $C^-(n)$}

Let us recall the set $C^-(n)$ defined in \cite[Section 2]{HKKMV}. For $n \geq 2$, let 
$$C^-(n):=\{(1,n-1)\} \cup \left\{(a,b) \in \ZZ^2 : 1 \leq a \leq b, \; a \leq \left\lfloor \frac{n}{2} \right\rfloor, \; b \leq n-2\right\}.$$
Then, by the proof of \cite[Theorem 2.9]{HKKMV}, we see that $C^-(n) \subset C^*(n)$. 
The difference of $C^-(n)$ and $C^*(n)$ will be mentioned below. See Example~\ref{ex:C^*}. 

\begin{rem}[{cf. \cite[Proposition 1.3]{HKKMV}}]\label{rem:star}
Let $G \in \Graph(n)$ and assume that $\alpha(G)=n-1$. 
Then $G$ itself must  be a bouquet. In particular, $G$ is chordal. Thus, by Theorem~\ref{thm:K2}, we have $\pdim G=n-1$. Hence, $\dep G=1$. 
These imply that if $(a,n-1) \in \gdd(n)$, then $a=1$. In particular, $(1,n-1) \in C^-(n) \subset C^*(n)$. 
\end{rem}

\begin{rem}\label{rem:C^-}
Given integers $a,b$ with $1 \leq a \leq b \leq n-1$, assume that $a+b \leq n$. 
Then $(a,b) \in C^-$ automatically holds. In fact, if $b=n-1$, then $a=1$; if $b \leq n-2$, since $2a \leq a+b \leq n$ holds, we have $a \leq \lfloor n/2 \rfloor$. 
\end{rem}

\bigskip

\section{In the case of graphs with large independence number}\label{sec:main1}

In this section, we consider the inclusion $\gdd(n) \subset C^*(n)$ for some special cases. 
In particular, we discuss when $(a,b) \in \gdd(n)$ belongs to $C^*(n)$. 

First, we discuss the case where $a=b$, i.e., the corresponding graph is Cohen--Macaulay. 
\begin{prop}[{cf. \cite[Theorem 3.3]{GV}}]\label{prop:CM}
Assume that $(b,b) \in \gdd(n)$. Then $(b,b) \in C^*(n)$. 
\end{prop}
\begin{proof}
We see that $(b,b) \in C^*(n)$ holds if and only if $\lceil b/(n-b) \rceil \leq 1$ holds. 
Here, it follows from \cite[Theorem 3.3]{GV} that if $G$ is Cohen--Macaulay, then $b \leq n/2$. 
This implies that when $G$ is Cohen--Macaulay with $b=\dim G$, we have $b \leq n/2$, i.e., $(b,b) \in C^*(n)$. 
\end{proof}

Next, we discuss the case where $b$ is relatively large. 
Note that we always have $1 \leq b \leq n-1$ and the case $b=n-1$ has been already discussed in Remark~\ref{rem:star}. 

\begin{prop}\label{prop:n-2}
Let $n \geq 4$. If $(a,n-2) \not\in C^*(n)$, then $(a,n-2) \not\in \gdd(n)$. 
\end{prop}
\begin{proof}
Let $$f(a,b):=b+1-\left\lceil \frac{b}{n-b} \right\rceil -a.$$ 
Note that for integers $a,b$ with $1 \leq a \leq b \leq n-1$, $f(a,b) \geq 0$ holds if and only if $(a,b) \in C^*(n)$. 

In what follows, we lead a contradiction if there exists a graph $G \in \Graph(n)$ with $(\dep G, \dim G)=(a,n-2)$ such that $f(a,n-2) < 0$. 

Since $f(a,n-2)=\lfloor n/2 \rfloor -a$, we have $\lfloor n/2 \rfloor < \dep G=n- \pdim G$, i.e., $\pdim G < \lceil n/2 \rceil$.  
On the other hand, since $\dim G=n-2$, there is an independent set $W$ of $G$ with $|W|=n-2$. 
Let $v_1,v_2$ be the remaining vertices, i.e., $\{v_1,v_2\}=V(G) \setminus W$. 
Since $G$ is connected and $W$ is an independent set, we see that 
\begin{align*}
n-2 &\leq \sum_{w \in W}\deg w = |N(v_1) \cap W|+|N(v_2) \cap W| \\
&= 
\begin{cases}
\deg v_1 + \deg v_2-2, &\text{ if }\{v_1,v_2\} \in E(G), \\
\deg v_1+\deg v_2, &\text{ if }\{v_1,v_2\} \not\in E(G).
\end{cases}
\end{align*}
On the other hand, since $\pdim G < \lceil n/2 \rceil$, we see that $\deg v_i < \lceil n/2 \rceil$ for $i=1,2$ by \eqref{eq:deg}. 
If $\{v_1,v_2 \} \in E(G)$, then we see that 
$$n-2 \leq \deg v_1 + \deg v_2 -2 \leq \frac{n-1}{2}+\frac{n-1}{2} - 2 = n-3,$$ a contradiction. 
Hence, $\{v_1,v_2\} \not\in E(G)$. 

Suppose that $n$ is even. Then we see that  
$$n-2 \leq \deg v_1 + \deg v_2 \leq \frac{n-2}{2} + \frac{n-2}{2} = n-2,$$
i.e., we have $\deg w = 1$ for each $w \in W$. Since $\{v_1,v_2\} \not\in E(G)$, we see that there is no path in $G$ connecting $v_1$ and $v_2$, 
a contradiction to the connectedness of $G$. 
Hence, $n$ should be odd. Moreover, there should be at least one vertex $w_0$ in $W$ with $\deg w_0 \geq 2$. 
Since $\sum_{w \in W} \deg w \geq n-1$ and $\deg v_i \leq (n-1)/2$ hold, we conclude that $G$ should look like as follows: 
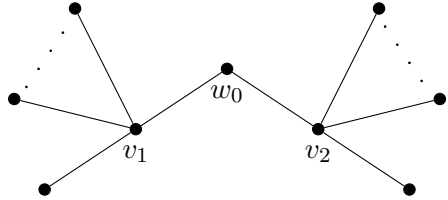
\begin{center}
\begin{figure}[h]
\begin{tikzpicture}[scale=0.8]
      \tikzset{enclosed/.style={draw, circle, inner sep=0pt, minimum size=.15cm, fill=black},close/.style={draw, circle, inner sep=0pt, minimum size=.02cm, fill=black}}

      \node[enclosed, label={below: $v_1$}] (A) at (4,0) {};
      \node[enclosed, label={below: $v_2$}] (B) at (7,0) {};
      \node[enclosed] (C) at (8,2) {};
      \node[enclosed] (D) at (3,2) {};
      \node[enclosed] (E) at (2,0.5) {};
      \node[enclosed] (F) at (9,0.5) {};
      \node[enclosed] (G) at (8.5,-1) {};
      \node[enclosed] (H) at (2.5,-1) {};
      \node[enclosed, label={below: $w_0$}] (I) at (5.5,1) {};
      \node[close] (J) at (2.2,0.7) {};
      \node[close] (K) at (2.4,1.0) {};
      \node[close] (P) at (2.6,1.3) {};
      \node[close] (L) at (2.8,1.6) {};
      \node[close] (M) at (8.1,1.7) {};
      \node[close] (N) at (8.3,1.4) {};
      \node[close] (O) at (8.5,1.1) {};
      \node[close] (Q) at (8.7,0.8) {};
        
      \draw (A) -- (I);
      \draw (B) -- (I);
      \draw (B) -- (C); 
      \draw (A) -- (D); 
      \draw (A) -- (E); 
      \draw (B) -- (F);
      \draw (B) -- (G);
      \draw (A) -- (H);
\end{tikzpicture}
\end{figure}
\end{center}
Note that $\deg v_1=\deg v_2=(n-1)/2$. 
In particular, $G$ is chordal. By taking a set of bouquets consisting of two bouquets whose roots are $v_1$ and $v_2$, 
we see from Theorem~\ref{thm:K2} that $$\left\lceil \frac{n}{2} \right\rceil > \pdim G \geq n-2,$$ 
a contradiction to $n \geq 4$. 
\end{proof}

\begin{prop}\label{prop:n-3}
Let $n \geq 8$. If $(a,n-3) \not\in C^*(n)$, then $(a,n-3) \not\in \gdd(n)$. 
\end{prop}
\begin{proof}
Let $f(a,b)$ be the same thing as above. Then $f(a,n-3)=\lfloor 2n/3 \rfloor -a-1$. 

Suppose that there exists a graph $G \in \Graph(n)$ with $(\dep G, \dim G)=(a,n-3)$ such that $f(a,n-3) < 0$. 
Since $\dim G=n-3$, there is an independent set $W$ with $|W|=n-3$. Let $v_1,v_2,v_3$ be the remaining vertices. 
By \eqref{eq:deg}, we have $\deg v_i \leq n-a < n- \lfloor 2n/3 \rfloor +1 = \lceil n/3 \rceil +1$, i.e., 
\begin{align}\label{eq:deg3}
\deg v_i \leq \left\lceil \frac{n}{3} \right\rceil \text{ for }i=1,2,3.
\end{align}
Similarly, for any strongly disjoint set $\calB$ of bouquets of $G$, we have 
\begin{align}\label{eq:bouquet}|F(\calB)| \leq \left\lceil \frac{n}{3} \right\rceil.\end{align}

Here, we have the following four possibilities on the adjacencies of $v_1,v_2,v_3$: 
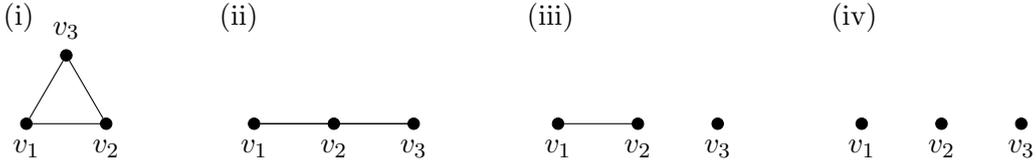
\begin{center}
\begin{tikzpicture}
\tikzset{enclosed/.style={draw, circle, inner sep=0pt, minimum size=.15cm, fill=black},scale=0.35}
\node[enclosed, label={above, xshift=-0.1cm, yshift=1cm:(i) }, label={below: $v_1$}] (A) at (4,0) {};
\node[enclosed, label={above, yshift=.2cm: }, label={below: $v_2$}] (B) at (7,0) {};
\node[enclosed, label={above, yshift=.2cm:}, label={above: $v_3$}] (C) at (5.5,2.59) {};

\draw (A) -- (B);
\draw (A) -- (C); 
\draw (B) -- (C); 
\end{tikzpicture}
\hspace{0.8cm}
\begin{tikzpicture}
\tikzset{enclosed/.style={draw, circle, inner sep=0pt, minimum size=.15cm, fill=black},scale=0.35}
\node[enclosed, label={below: $v_1$}, label={above, xshift=-0.2cm, yshift=1cm:(ii)}] (A) at (4,0) {};
\node[enclosed, label={below: $v_2$}] (B) at (7,0) {};
\node[enclosed, label={below: $v_3$}] (C) at (10,0) {};
\draw (A) -- (B);
\draw (A) -- (C); 
\draw (B) -- (C); 
\end{tikzpicture}
\hspace{0.8cm}
\begin{tikzpicture}
\tikzset{enclosed/.style={draw, circle, inner sep=0pt, minimum size=.15cm, fill=black},scale=0.35}
\node[enclosed, label={above, xshift=-0.1cm,yshift=1cm:(iii) }, label={below: $v_1$}] (A) at (4,0) {};
\node[enclosed, label={above, yshift=.2cm: }, label={below: $v_2$}] (B) at (7,0) {};
\node[enclosed, label={above, yshift=.2cm:}, label={below: $v_3$}] (C) at (10,0) {};
\draw (A) -- (B);
\end{tikzpicture}
\hspace{0.8cm}
\begin{tikzpicture}
\tikzset{enclosed/.style={draw, circle, inner sep=0pt, minimum size=.15cm, fill=black},scale=0.35}
\node[enclosed, label={above, xshift=-0.1cm, yshift=1cm:(iv) }, label={below: $v_1$}] (A) at (4,0) {};
\node[enclosed, label={above, yshift=.2cm: }, label={below: $v_2$}] (B) at (7,0) {};
\node[enclosed, label={above, yshift=.2cm:}, label={below: $v_3$}] (C) at (10,0) {};
\end{tikzpicture}
\end{center}
We divide our discussions into these four cases. 

(i) In this case, we see from \eqref{eq:deg3} that 
$$n-3 \leq \sum_{w \in W} \deg w = \sum_{i=1}^3 (\deg v_i -2) \leq 3\left( \left\lceil \frac{n}{3} \right\rceil - 2\right) \leq 3 \left( \frac{n+2}{3}-2\right) = n-4,$$ 
a contradiction. 

(ii) By the similar computation to (i), we have that 
\begin{align*}
n-3 \leq \sum_{w \in W} \deg w = \sum_{i=1}^3 \deg v_i - 4 \leq 3 \left\lceil \frac{n}{3} \right\rceil - 4 = \begin{cases}
n-4 \;\;\text{ if }n \equiv 0 \pmod{3}, \\
n-3 \;\;\text{ if }n \equiv 2 \pmod{3}, \\ 
n-2 \;\;\text{ if }n \equiv 1 \pmod{3}. 
\end{cases}
\end{align*}
Thus, we get the following possibilities: 
\begin{itemize}
\item $n \equiv 2 \pmod{3}$ and $\deg w = 1$ for each $w \in W$; 
\item $n \equiv 1 \pmod{3}$ and $\deg w = 1$ for each $w \in W$; 
\item $n \equiv 1 \pmod{3}$, $\deg w_0=2$ for some $w_0 \in W$ and $\deg w = 1$ for each $w \in W \setminus \{w_0\}$.  
\end{itemize}
In the first two cases (resp. the third case), $G$ should look like the left-most one (resp. the central one or the right-most one) of the following figure: 
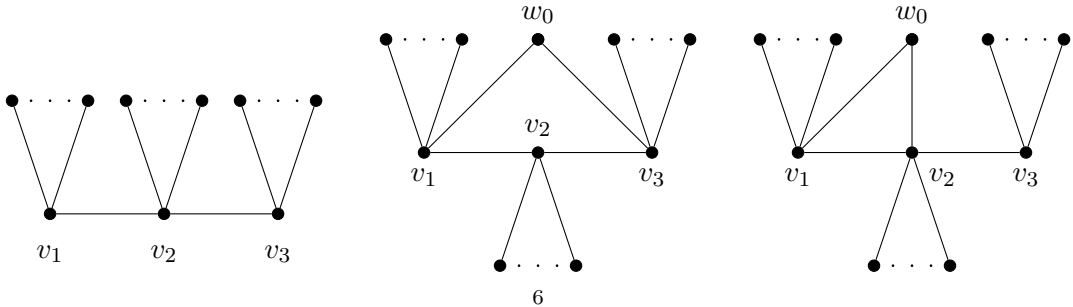
\begin{center}
\begin{tikzpicture}
      \tikzset{enclosed/.style={draw, circle, inner sep=0pt, minimum size=.15cm, fill=black},close/.style={draw, circle, inner sep=0pt, minimum size=.02cm, fill=black},scale=0.5}
      \node[enclosed, label={below, yshift=-.2cm: $v_1$}] (A) at (4,0) {};
      \node[enclosed, label={below, yshift=-.2cm: $v_2$}] (B) at (7,0) {};
      \node[enclosed, label={below, yshift=-.2cm:$v_3$}] (C) at (10,0) {};
      \node[enclosed] (D) at (3,3) {};
      \node[enclosed] (E) at (5,3) {};
      \node[enclosed] (F) at (6,3) {};
      \node[enclosed] (G) at (8,3) {};
      \node[enclosed] (H) at (9,3) {};
      \node[enclosed] (I) at (11,3) {};
       \node[close] (J) at (3.5,3) {};
      \node[close] (K) at (4,3) {};
      \node[close] (L) at (4.5,3) {};
      \node[close] (M) at (6.5,3) {};
      \node[close] (N) at (7,3) {};
      \node[close] (O) at (7.5,3) {};
      \node[close] (P) at (9.5,3) {};
      \node[close] (Q) at (10,3) {};
      \node[close] (R) at (10.5,3) {};

      \draw (A) -- (D);
      \draw (A) -- (E); 
      \draw (B) -- (F); 
      \draw (B) -- (G); 
      \draw (C) -- (H);
      \draw (C) -- (I);
      \draw (A) -- (B);
      \draw (B) -- (C);
\end{tikzpicture}
\hspace{0.5cm}
\begin{tikzpicture}
      \tikzset{enclosed/.style={draw, circle, inner sep=0pt, minimum size=.15cm, fill=black},close/.style={draw, circle, inner sep=0pt, minimum size=.02cm, fill=black},scale=0.5}

      \node[enclosed, label={below: $v_1$}] (A) at (4,0) {};
      \node[enclosed, label={above: $v_2$}] (B) at (7,0) {};
      \node[enclosed, label={below: $v_3$}] (C) at (10,0) {};
      \node[enclosed, label={above: $w_0$}] (W) at (7,3) {};
      \node[enclosed] (D) at (3,3) {};
      \node[enclosed] (E) at (5,3) {};
      \node[enclosed] (F) at (6,-3) {};
      \node[enclosed] (G) at (8,-3) {};
      \node[enclosed] (H) at (9,3) {};
      \node[enclosed] (I) at (11,3) {};
      \node[enclosed] (J) at (7,3) {};
      \node[close] (K) at (3.5,3) {};
      \node[close] (L) at (4,3) {};
      \node[close] (M) at (4.5,3) {};
      \node[close] (N) at (6.5,-3) {};
      \node[close] (O) at (7,-3) {};
      \node[close] (P) at (7.5,-3) {};
      \node[close] (Q) at (9.5,3) {};
      \node[close] (R) at (10,3) {};
      \node[close] (S) at (10.5,3) {};
      
      \draw (A) -- (D);
      \draw (A) -- (E); 
      \draw (B) -- (F); 
      \draw (B) -- (G); 
      \draw (C) -- (H);
      \draw (C) -- (I);
      \draw (A) -- (B);
      \draw (B) -- (C);
      \draw (A) -- (J);
      \draw (J) -- (C);
\end{tikzpicture}
\hspace{0.5cm}
\begin{tikzpicture}
      \tikzset{enclosed/.style={draw, circle, inner sep=0pt, minimum size=.15cm, fill=black},close/.style={draw, circle, inner sep=0pt, minimum size=.02cm, fill=black},scale=0.5}

      \node[enclosed, label={below: $v_1$}] (A) at (4,0) {};
      \node[enclosed, label={below, xshift=0.4cm: $v_2$}] (B) at (7,0) {};
      \node[enclosed, label={below,:$v_3$}] (C) at (10,0) {};
      \node[enclosed, label={above: $w_0$}] (W) at (7,3) {};
      \node[enclosed] (D) at (3,3) {};
      \node[enclosed] (E) at (5,3) {};
      \node[enclosed] (F) at (6,-3) {};
      \node[enclosed] (G) at (8,-3) {};
      \node[enclosed] (H) at (9,3) {};
      \node[enclosed] (I) at (11,3) {};
      \node[enclosed] (J) at (7,3) {};
      \node[close] (K) at (3.5,3) {};
      \node[close] (L) at (4,3) {};
      \node[close] (M) at (4.5,3) {};
      \node[close] (N) at (6.5,-3) {};
      \node[close] (O) at (7,-3) {};
      \node[close] (P) at (7.5,-3) {};
      \node[close] (Q) at (9.5,3) {};
      \node[close] (R) at (10,3) {};
      \node[close] (S) at (10.5,3) {};
      
      \draw (A) -- (D);
      \draw (A) -- (E); 
      \draw (B) -- (F); 
      \draw (B) -- (G); 
      \draw (C) -- (H);
      \draw (C) -- (I);
      \draw (A) -- (B);
      \draw (B) -- (C);
      \draw (A) -- (J);
      \draw (J) -- (B);
\end{tikzpicture}
\end{center}
Note that we have $|N(v_i) \cap W| \geq \lceil n/3 \rceil -2 > 1$ for $i=1,3$ and $|(N(v_1) \cup N(v_3)) \cap W| \geq 2 \lceil n/3 \rceil -3$ in any cases. 
Hence, we can always take a strongly disjoint set of bouquets $\calB$ consisting of two bouquets whose roots are $v_1$ and $v_3$ 
satisfying $|F(\calB)| \geq 2 \lceil n/3 \rceil -3+1.$
(Note that ``$+1$'' comes from a stem $\{v_1,v_2\}$.) By \eqref{eq:bouquet}, we obtain that 
$$\left\lceil \frac{n}{3} \right\rceil \geq 2 \left\lceil \frac{n}{3} \right\rceil -2 \;\Longleftrightarrow\; 2 \geq \left\lceil \frac{n}{3} \right\rceil 
\;\Longleftrightarrow\; 6 \geq n,$$
a contradiction. 

(iii) Take bouquets $B_1,B_2,B_3$ of $G$ whose roots are $v_1,v_2,v_3$, respectively, such that 
$V(B_i) \cap V(B_j)=\emptyset$ ($i \neq j$) and $F(\{B_1,B_2,B_3\})=W$ hold. 
\begin{itemize}
\item Assume that $N(v_3) \setminus N(v_2) \neq \emptyset$. Let $w \in N(v_3) \setminus N(v_2)$. 
Then, by taking a strongly disjoint set $\calB'$ of bouquets consisting of $B_2'$ and a bouquet $\{v_3,w\}$, 
where $B_2'$ is constructed by attaching a new flower $v_1$ to $B_2$, we see from \eqref{eq:bouquet} that \begin{align*}
\left\lceil \frac{n}{3} \right\rceil \geq |F(\calB')| = |F(B_2)|+2 = n-3-|F(\{B_1,B_3\})|+2 \geq n-1-\left\lceil \frac{n}{3} \right\rceil. 
\end{align*}
Hence, we obtain the inequality $$\left\lceil \frac{n}{3} \right\rceil \geq n-1-\left\lceil \frac{n}{3} \right\rceil,$$ i.e., $n \leq 7$, a contradiction. 
\item Assume that $N(v_3) \setminus N(v_2)=\emptyset$, i.e., $N(v_3) \subset N(v_2)$. 
Take $w \in N(v_3)$ and let $B'$ be a new bouquet with its root $w$ and flowers $v_2,v_3$. 
Then $\{B_1,B'\}$ becomes strongly disjoint. Thus, by \eqref{eq:bouquet}, we have 
$$\left\lceil \frac{n}{3} \right\rceil \geq |F(\{B_1,B'\})| = |F(B_1)|+2 \geq n-3-|F(\{B_2\})|+2 \geq n-1-\left\lceil \frac{n}{3} \right\rceil,$$ 
a contradiction. 
\end{itemize}

(iv) Take bouquets $B_1,B_2,B_3$ of $G$ whose roots are $v_1,v_2,v_3$, respectively, such that 
$V(B_i) \cap V(B_j)=\emptyset$ ($i \neq j$) and $F(\{B_1,B_2,B_3\})=W$ hold. 
If the situation (iv) happens, then there are $j,k$ such that $\{B_j,B_k\}$ is strongly disjoint. 
In fact, if not, then either $F(B_j) \subset F(B_k)$ or $F(B_k) \subset F(B_j)$ holds for each $1 \leq j<k \leq 3$. 
This implies that we have $F(B_1) \subset F(B_2) \subset F(B_3)$ after properly changing the indices of $B_1,B_2,B_3$. 
In particular, $F(B_3)=W$. Thus, $n-3=\deg v_3 \leq \lceil n/3 \rceil \leq (n+2)/3$, i.e., $n \leq 5$, a contradiction. 

Let, say, $\calB=\{B_1,B_2\}$ be a strongly disjoint set of bouquets of $G$. 
\begin{itemize}
\item Suppose that $|F(\calB)| \geq \lfloor 2n/3 \rfloor -1$. Then we see that 
$$\left\lceil \frac{n}{3} \right\rceil \geq |F(\calB)| \geq \left\lfloor \frac{2n}{3} \right\rfloor -1 \;\Longrightarrow\;
\frac{n+2}{3} \geq \frac{2n-2}{3} -1 \;\Longleftrightarrow\; 7 \geq n,$$
a contradiction. 
\item Suppose that $|F(\calB)| \leq \lfloor 2n/3 \rfloor -2$. Since $|W|=|F(\calB)|+|F(B_3)|=n-3$ and $|F(B_3)| \leq \deg v_3-1$ 
by the connectedness of $G$, we see that 
$$\left\lceil \frac{n}{3} \right\rceil - 1 \geq \deg v_3 - 1\geq |F(B_3)| \geq n-3 - \left\lfloor \frac{2n}{3} \right\rfloor +2 =\left\lceil \frac{n}{3} \right\rceil -1. $$
Hence, all the equalities of these inequalities are satisfied. In particular, we have $|F(B_3)|= \deg v_3 -1 = \lceil n/3 \rceil-1$. 
Since $n-3 > \lceil n/3 \rceil$ and $G$ is connected, there is $w \in W \setminus N(v_3)$ adjacent to $v_1$ or $v_2$, say, $v_1$.  
Then we can construct a new strongly disjoint set of bouquets $\calB'$ consisting of bouquets $N(v_3) \cup \{v_3\}$ and $\{v_1,w\}$ 
such that $|F(\calB')| = \lceil n/3 \rceil +1$, a contradiction to \eqref{eq:bouquet}.  
\end{itemize}
\end{proof}

By using Propositions~\ref{prop:CM}, \ref{prop:n-2} and \ref{prop:n-3}, we can prove Corollary~\ref{cor:small}. 
\begin{proof}[Proof of Corolalry~\ref{cor:small}]
It is enough to show that $\gdd(n) \subset C^*(n)$ for $n \leq 12$. 

Take $(a,b) \in \gdd(n)$ arbitrarily. Then the inequalities $1 \leq a \leq b \leq n-1$ automatically hold. 
\begin{itemize}
\item If $a=b$, then we see that $(a,b) \in C^*(n)$ by Proposition~\ref{prop:CM}. 
\item The case of $b=n-1$ has been already discussed in Remark~\ref{rem:star}. 
\end{itemize}

Hence, in what follows, we assume that $1 \leq a<b \leq n-2$. 
Here, we consider the set $$C'(n):=\{(a,b) \in \ZZ^2 : 1 \leq a < b \leq n-2\} \setminus C^*(n).$$ 
Then the direct computations show the following: 
\begin{align*}
&C'(n)=\emptyset \text{ for }2 \leq n \leq 6, \;\; C'(7)=\{(4,5)\}, \;\; C'(8)=\{(5,6)\}, \;\; C'(9)=\{(5,7),(6,7)\}, \\
&C'(10)=\{(6,7),(6,8),(7,8)\}, \;\; C'(11)=\{(6,9),(7,8),(7,9),(8,9)\}, \\ 
&C'(12)=\{(7,10),(8,9),(8,10),(9,10)\}. 
\end{align*}

Let $n \leq 12$. 
Suppose that $(a,b) \not\in C^*(n)$. Then $(a,b) \in C'(n)$. 
By the above computations, we see that all elements of $\bigcup_{n=2}^{12} C'(n)$ satisfy $b=n-2$ or $b=n-3$. 
Hence, it follows from Propositions~\ref{prop:n-2} and \ref{prop:n-3} that $(a,b) \not\in \gdd(n)$, a contradiction. 

Therefore, $C^*(n)=\gdd(n)$ holds if $n \leq 12$. 
\end{proof}

\begin{rem}
We see that $$C'(13)=\{(7,11),(8,9),(8,10),(8,11),(9,10),(9,11),(10,11)\}.$$ 
Since we cannot claim that $(8,9) \not\in \gdd(13)$ only by Propositions~\ref{prop:n-2} and \ref{prop:n-3}, 
the above proof of Corollary~\ref{cor:small} is not available if $n=13$. 
Instead, we can claim that $\gdd(13) = C^*(13)$ or $C^*(13) \cup \{(8,9)\}$. 
\end{rem}

\begin{ex}\label{ex:C^*}
For small $n$'s, the direct computations show the following: 
\begin{align*}
C^*(n)&=C^-(n) \text{ for }n=2,3,\ldots,8,10, \\ 
C^*(9)&=C^-(9) \cup \{(5,6)\}, \\ 
C^*(11)&=C^-(11) \cup \{(6,6),(6,7),(6,8)\}, \\ 
C^*(12)&=C^-(12) \cup \{(7,8),(7,9),(8,9)\}. 
\end{align*}
Note that \cite[Example 2.11]{HKKMV} also mentions the similar equalities in the case $3 \leq n \leq 9$. 
\end{ex}

\bigskip

\section{In the case of chordal graphs}\label{sec:main2}
In this section, we discuss the behavior of pairs of dimensions and depths of edge ideals of chordal graphs.

\subsection{The inclusion $C^*(n) \subset \gddcc(n)$}

First, we discuss the inclusion $C^*(n) \subset \gddcc(n)$. 
Here, we remark that since we only know that $\gddcc(n) \subset \gdd(n)$ by definition, 
we cannot immediately claim $C^*(n) \subset \gddcc(n)$ from Theorem~\ref{thm:inc}. 

\medskip

Let $(a,b) \in C^*(n)$. We divide the discussions into two cases: 
\begin{itemize}
\item[(i)] $a+b > n$; or 
\item[(ii)] $a+b \leq n$. 
\end{itemize}

(i) In this case, we can apply the same proof as that of \cite[Theorem 2.8]{HKKMV}. 
In that proof, the graphs $G(m;s_1,\ldots,s_m)$ given in \cite[Construction 2.5]{HKKMV} play the essential role 
and we see that $G(m;s_1,\ldots,s_m)$ is chordal. Hence, we obtain the inclusion by the same discussion. 

(ii) In this case, we cannot apply the same proof. In fact, Lemmas 2.3 and 2.4 are used for the proof in the case $a+b \leq n$. 
Especially, Lemma 2.3 uses the technique of {\em $S$-cone} (also known as $S$-suspension). 
In general, the chordality is not preserved by taking $S$-cone. Hence, Lemma 2.3 is not immediately available for chordal graphs. 

Instead, we directly prove the following: 
\begin{prop}
Let $a,b,n$ be integers satisfying $1 \leq a \leq b \leq n-1$ and $a+b \leq n$. 
Then there exists a chordal graph $G \in \Graph(n)$ such that $(\dep G,\dim G) = (a,b)$. 
Namely, we have $(a,b) \in \gddcc(n)$. 
\end{prop}
Note that the assumption $a+b \leq n$ immediately implies $(a,b) \in C^*(n)$ (see Remark~\ref{rem:C^-}). 
\begin{proof}
Let $G$ be the graph defined as follows: 
\begin{align*}
V(G)&=U \sqcup W, \;\text{ where }\; U=\{u_1,\ldots,u_{n-b}\} \text{ and }W=\{w_1,\ldots,w_b\}; \\
E(G)&=\{\{u,u'\} : u,u' \in U\} \cup \{\{u_i,w_i\} : 1 \leq i \leq a-1\} \\
&\quad\quad\quad\quad\quad\quad\quad\quad\quad\quad\cup \{\{u_j,w_k\} : a \leq j \leq n-b, \; a \leq k \leq b\}. 
\end{align*}
It is straightforward to check that the ordering $w_1,\ldots,w_b,u_1,\ldots,u_{n-b}$ becomes a perfect elimination ordering of $G$. Hence, $G$ is chordal. 
Moreover, we can easily see that $W$ is an independent set with $|W|=\alpha(G)$. Hence, $\dim G=b$. 

In what follows, we prove that $\dep G=a$, i.e., $\pdim G=n-a$. 

Regarding the minimal vertex covers of $G$, we observe that any minimal vertex cover contains at least $(n-b-1)$ vertices of $U$ since $U$ forms a clique. 
In particular, we see that $U$ itself becomes a minimal vertex cover. 
Fix $u_j \in U$ and let $C$ be a minimal vertex cover not containing $u_j$. Then $U \setminus \{u_j\} \subset C$. 
\begin{itemize}
\item If $1 \leq j \leq a-1$, then $w_j \in C$, so $C=(U\setminus \{u_j\}) \cup \{w_j\}$. 
\item If $a \leq j \leq n-b$, then $w_a,\ldots,w_b \in C$, so $C=(U\setminus \{u_j\}) \cup \{w_a,\ldots,w_b\}$. 
\end{itemize}
Therefore, by Theorem~\ref{thm:K2}, we conclude that 
$$\pdim G= \max\{|C| : C\text{ is a minimal vertex cover of }G\}=\max\{|U|,|U|+b-a\}=n-a.$$
\end{proof}

\subsection{The inclusion $\gddcc(n) \subset C^*(n)$}

For the proof of this inclusion, we prepare the following lemma. 
\begin{lem}\label{lem:chordal}
{\em (1)} For any vertex $v$ of a graph $G$, we have 
$$\dim G = \max\{\dim (G \setminus N[v])+1, \dim (G\setminus v)\}.$$
{\em (2)} Let $G$ be a chordal graph and let $v_n,\ldots,v_1$ be a perfect elimination ordering of $G$. Then we have 
$$\dep G =\min\{ \dep (G \setminus N[v_n])+1, \dep (G \setminus v_n)\}.$$
\end{lem}
\begin{proof}
(1) Let $W$ be an independent set of $G$ with $|W|=\alpha(G)$. 
Then the required equality directly follows from the following observations: 
\begin{itemize}
\item If $v \in W$, then $W \setminus \{v\} \subset V(G \setminus N[v])$ becomes an independent set of $G \setminus N[v]$ 
with $|W \setminus \{v\}|=\alpha(G \setminus N[v])$. 
\item If $v \not\in W$, then $W \subset V(G \setminus v)$ becomes an independent set of $G \setminus v$ with $|W|=\alpha(G \setminus v)$. 
\end{itemize}

\medskip

\noindent
(2) We prove the equality by induction on $|V(G)|$. 

Let $v=v_n$ and $\deg v=m$. Let $C$ be a minimal vertex cover of $G$ with $|C|=\pdim G$ (cf. Theorem~\ref{thm:K2}). 
\begin{itemize}
\item If $v \in C$, then $C'=C \setminus \{v\}$ becomes a minimal vertex cover of $G \setminus v$ with $\pdim (G \setminus v)=|C'|$. 
Hence, by the hypothesis of induction, we see that \begin{align*}
\dep G&=|V(G)|-\pdim G=|V(G)|-(|C'|+1)=|V(G \setminus v)| - |C'| \\
&= |V(G \setminus v)| - \pdim (G \setminus v) = \dep (G \setminus v). 
\end{align*}
\item If $v \not\in C$, since $N[v]$ forms a clique by definition of perfect elimination ordering, we see that $N(v) \subset C$. 
Thus, $C''=C \setminus N[v]$ becomes a minimal vertex cover of $G \setminus N[v]$ with $\pdim (G \setminus N[v])=|C''|$. 
Hence, we see that \begin{align*}
\dep G&=|V(G)|-\pdim G=|V(G)|-(|C''|+m)=(|V(G \setminus N[v])|+1) - |C''| \\
&= |V(G \setminus N[v])| - \pdim (G \setminus N[v])+1 = \dep (G \setminus N[v])+1. 
\end{align*}
\end{itemize}
Here, Theorem~\ref{thm:K2} says that if $C_0$ is a minimal vertex cover with $|C_0| \neq \pdim G$, then $|C_0|<\pdim G$. 
Thus, the required equality follows. 
\end{proof}

Now, we prove the inclusion $\gddcc(n) \subset C^*(n)$. Take $(a,b) \in \gddcc(n)$ arbitrarily. 
Then there is a chordal graph $G$ such that $(\dep G,\dim G)=(a,b)$. 
To prove $(a,b) \in C^*(n)$, it is enough to show that \begin{align}\label{goal}b \leq (n-b)(b-a+1). \end{align}
In fact, \eqref{goal} is equivalent to the following: 
\begin{align*}
\frac{b}{n-b} \leq b-a+1 \;\Longleftrightarrow\; \left\lceil\frac{b}{n-b}\right\rceil \leq b-a+1 \;\Longleftrightarrow\; 
a \leq b+1 - \left\lceil \frac{b}{n-b} \right\rceil. 
\end{align*}
This means that $(a,b) \in C^*(n)$. 

Let $v_n,\ldots,v_1$ be a perfect elimination ordering of $G$. We prove the inequality \eqref{goal} by induction on $n$. 

Let $v=v_n$ and let $$a'=\dep (G \setminus v), \;\; b'=\dim (G \setminus v), \;\; a''=\dep (G \setminus N[v]) \text{ and } b''=\dim (G \setminus N[v]).$$ 
By Lemma~\ref{lem:chordal} (1), either $b = b'$ or $b = b''+1$ holds. 
\begin{itemize}
\item Assume that $b = b'$. Since $a \leq a'$ holds by Lemma~\ref{lem:chordal} (2), we see from the hypothesis of induction that 
\begin{align*}
b=b' \leq ((n-1)-b')(b'-a'+1) \leq (n-b)(b-a+1). 
\end{align*}
\item Assume that $b = b''+1$. Since $a \leq a''+1$, we see that 
\begin{align*}
b&=b''+1 \leq ((n-\deg v -1)-b'')(b''-a''+1)+1 \\
&\leq (n-\deg v - b)(b-a+1)+1 = (n- b)(b-a+1) -(\deg v(b-a+1)-1) \\
&\leq (n-b)(b-a+1). 
\end{align*}
\end{itemize}


\end{document}